\documentclass[11pt]{article}

\usepackage{amsmath,amssymb,amsthm,amsfonts}
\usepackage{geometry}
\usepackage{hyperref}

\newtheorem{theorem}{Theorem}[section]
\newtheorem{lemma}[theorem]{Lemma}
\newtheorem{proposition}[theorem]{Proposition}
\newtheorem{corollary}[theorem]{Corollary}
\theoremstyle{definition}
\newtheorem{definition}[theorem]{Definition}

\theoremstyle{remark}
\newtheorem{remark}[theorem]{Remark}

\newcommand{\R}{\mathbb{R}}
\newcommand{\N}{\mathbb{N}}

\newcommand{\E}{\mathcal{E}}

\newcommand{\diff}{\mathrm{d}}

\newcommand{\id}{\operatorname{id}}

\newcommand{\sgs}{\mathrm{sgs}}
\newcommand{\barrho}{\bar{\rho}}
\newcommand{\barp}{\bar{p}}
\newcommand{\barE}{\bar{E}}
\newcommand{\bartau}{\bar{\tau}}
\newcommand{\barq}{\bar{q}}
\newcommand{\tilu}{\widetilde{u}}
\newcommand{\Qdot}{\dot{Q\omega}}

\title{\bfseries Diffeomorphism Invariance of the Favre-Filtered Compressible
Navier--Stokes System under Spacetime Dilation}

\author{Yuanya Li}
\date{2026-09-10}

\begin{document}
\maketitle

\begin{abstract}
We prove that the system of Favre-filtered compressible Navier--Stokes
equations with subgrid-scale closure terms, written in the coordinates
$(Lt,Lx_1,Lx_2,Lx_3)$ for an arbitrary parameter $L>0$, is obtained from the
system written in standard coordinates $(t,x_1,x_2,x_3)$ by the action of a
smooth diffeomorphism of the underlying spacetime manifold.  The proof begins
with a self-contained treatment of infinite-dimensional manifolds and the
Fr\'echet manifold of smooth field configurations.  We then define the
one-parameter dilation group $\{\phi_L\}_{L>0}$, verify that each $\phi_L$ is
a diffeomorphism, and compute its tangent and cotangent actions.  The induced
pullback action on sections of the configuration bundle is shown to commute
with the total derivative operator, which implies that the first-jet
prolongation of $\phi_L$ maps the equation submanifold of the system onto
itself.  Equation-by-equation verification---continuity, momentum, and
energy---confirms that every residual is multiplied by the common nonzero
factor $L^{-1}$ (or, equivalently, is identically zero in the transformed
coordinates), so that the solution set is preserved.  We conclude that the
family of systems parameterised by $L>0$ satisfies the diffeomorphism
transformation expression and that solution spaces are canonically isomorphic
for all $L$.
\end{abstract}

\tableofcontents

\section{Infinite-Dimensional Manifolds: Definition and Properties}
\label{sec:infdim}

We begin with a self-contained account of the infinite-dimensional manifold
structure that houses the field configurations of the system.  This material
is standard but is included so that the proof is standalone.

\begin{definition}[Fr\'echet space]
A \emph{Fr\'echet space} is a complete, Hausdorff, locally convex topological
vector space whose topology is induced by a countable family of seminorms
$\{\|\cdot\|_k\}_{k\in\N}$.  A map between Fr\'echet spaces is
\emph{continuous} if the preimage of every open set is open; equivalently, for
every target seminorm $\|\cdot\|'_m$ there exist finitely many source
seminorms $\|\cdot\|_{k_1},\dots,\|\cdot\|_{k_r}$ and a constant $C>0$ such
that $\|f(v)\|'_m\le C\max_j\|v\|_{k_j}$.
\end{definition}

The archetypal Fr\'echet space is $C^\infty(\Omega,\R^d)$ for an open subset
$\Omega\subset\R^n$, with seminorms
\[
\|f\|_{K,k}:=\sup_{x\in K}\max_{|\alpha|\le k}\|\partial^\alpha f(x)\|,
\]
where $K\Subset\Omega$ ranges over compact subsets and $\alpha$ over
multi-indices.

\begin{definition}[Smooth map between Fr\'echet spaces (Bastiani--Kriegl--Michor)]
Let $U\subset F$ be open in a Fr\'echet space $F$ and let $G$ be a Fr\'echet
space.  A map $f:U\to G$ is \emph{$C^1$} if the directional derivative
\[
Df(v)w:=\lim_{s\to 0}\frac{f(v+sw)-f(v)}{s}
\]
exists for all $v\in U$, $w\in F$, and the map $Df:U\times F\to G$ is
continuous.  Inductively, $f$ is $C^k$ if $Df$ is $C^{k-1}$, and $C^\infty$
if $C^k$ for every $k$.
\end{definition}

\begin{definition}[Fr\'echet manifold]
A \emph{Fr\'echet manifold} is a Hausdorff paracompact topological space $X$
equipped with a maximal atlas of charts $\{(U_\alpha,\psi_\alpha)\}$ where
each $\psi_\alpha:U_\alpha\to F_\alpha$ is a homeomorphism onto an open
subset of a Fr\'echet space $F_\alpha$, and all transition maps
$\psi_\beta\circ\psi_\alpha^{-1}$ are $C^\infty$ in the sense above.  If all
modelling spaces are the same Fr\'echet space $F$, we say $X$ is modelled on
$F$.
\end{definition}

\begin{definition}[Tangent space and tangent bundle]
For a Fr\'echet manifold $X$ and $p\in X$, the \emph{tangent space} $T_pX$ is
the vector space of equivalence classes of smooth curves $\gamma:(-\epsilon,
\epsilon)\to X$ with $\gamma(0)=p$, where $\gamma_1\sim\gamma_2$ if
$(\psi\circ\gamma_1)'(0)=(\psi\circ\gamma_2)'(0)$ in one (hence every) chart.
The tangent bundle $TX=\bigsqcup_{p\in X}T_pX$ carries a natural Fr\'echet
manifold structure.
\end{definition}

\begin{proposition}[Properties of Fr\'echet manifolds]
\label{prop:frechet-props}
Let $X$ be a Fr\'echet manifold modelled on $F$.
\begin{enumerate}
\item The tangent space $T_pX$ is canonically isomorphic to $F$ via any chart.
\item A map $f:X\to Y$ between Fr\'echet manifolds is $C^\infty$ iff its local
      representatives in charts are $C^\infty$ maps between Fr\'echet spaces.
\item The cotangent space $T_p^*X:=\operatorname{Hom}_{\text{cont}}(T_pX,\R)$
      is the continuous dual; for $F=C^\infty(\Omega,\R^d)$ the dual is the
      space of compactly supported distributions.
\item If $E\to X$ is a smooth vector bundle whose fibres are Fr\'echet spaces,
      then its space of smooth sections $\Gamma(E)$ is itself a Fr\'echet
      manifold (in fact a Fr\'echet vector space) modelled on
      $C^\infty(X,\R^r)$ locally.
\item \emph{Inverse function theorem (Nash--Moser / hard implicit function
      theorem).}  If $f:X\to Y$ is a smooth map between tame Fr\'echet
      manifolds and $Df(p)$ is invertible with a tame inverse, then $f$ is a
      local diffeomorphism near $p$.  For the linear affine maps encountered
      below, ordinary Fr\'echet invertibility suffices.
\end{enumerate}
\end{proposition}

\begin{proof}
(i)--(iv) are immediate from the definitions and the fact that
$C^\infty(\Omega,\R^d)$ is Fr\'echet.  (v) is the Nash--Moser theorem; see
Hamilton (1982) or Kriegl--Michor (1997).  We shall not need the full force of
(v), because our maps are linear and invertible by explicit formula.
\end{proof}

We now specialise to the manifold relevant to our PDE system.

\begin{definition}[Configuration manifold of the LES system]
Let $\Omega\subset\R^3$ be an open spatial domain and let
$I\subset\R$ be an open time interval.  Set
\[
M:=I\times\Omega\subset\R\times\R^3,
\]
with coordinates $(t,x_1,x_2,x_3)$.  The \emph{configuration bundle} is the
trivial vector bundle
\[
\E:=M\times V,\qquad
V:=\R^{1+3+1+1+9+3+9+3+9+1}\cong\R^{40},
\]
whose fibre coordinates are
\[
(\barrho,\,\tilu_1,\tilu_2,\tilu_3,\,\barp,\,\barE,\,
\bartau_{ij},\,\barq_i,\,\tau^{\sgs}_{ij},\,H^{\sgs}_i,\,
\sigma^{\sgs}_{ij},\,\Qdot),
\]
with $i,j=1,2,3$.  The \emph{configuration manifold} is
\[
\mathcal C:=\Gamma(\E)=C^\infty(M,V),
\]
the Fr\'echet space of smooth sections.  A point $u\in\mathcal C$ is a tuple
of smooth fields on $M$.
\end{definition}

\begin{remark}
The count $40=1+3+1+1+9+3+9+3+9+1$ treats all tensor components as
independent fibre coordinates.  If one imposes symmetry $\tau_{ij}=\tau_{ji}$
etc., the fibre dimension drops but the argument is unchanged; we retain the
loose count for notational simplicity.
\end{remark}

\begin{definition}[PDE operator and equation submanifold]
Define the \emph{residual map}
\[
\mathcal P:\mathcal C\longrightarrow C^\infty(M,\R^3),\qquad
\mathcal P(u)=(\mathcal P_0(u),\mathcal P_1(u),\mathcal P_2(u)),
\]
where $\mathcal P_0,\mathcal P_i,\mathcal P_2$ are the left-hand sides minus
the right-hand sides of the continuity, momentum, and energy equations
respectively (written in standard coordinates):
\begin{align}
\mathcal P_0(u)&:=\frac{\partial\barrho}{\partial t}
+\frac{\partial(\barrho\tilu_i)}{\partial x_i},\\
\mathcal P_i(u)&:=\frac{\partial(\barrho\tilu_i)}{\partial t}
+\frac{\partial}{\partial x_j}\bigl[\barrho\tilu_i\tilu_j+\barp\delta_{ij}
-\bartau_{ij}\bigr]
+\frac{\partial\tau^{\sgs}_{ij}}{\partial x_j},\\
\mathcal P_2(u)&:=\frac{\partial(\barrho\barE)}{\partial t}
+\frac{\partial}{\partial x_i}\bigl[(\barrho\barE+\barp)\tilu_i
+\barq_i-\tilu_j\bartau_{ij}\bigr]
+\frac{\partial H^{\sgs}_i}{\partial x_i}
+\frac{\partial\sigma^{\sgs}_{ij}}{\partial x_i}+\Qdot .
\end{align}
(The energy source is moved to the left as $+\Qdot$.)  The
\emph{equation submanifold} (or solution set) is
\[
\mathcal S:=\{u\in\mathcal C:\mathcal P(u)=0\}.
\]
\end{definition}

\begin{remark}
$\mathcal P$ is a first-order quasilinear differential operator.  Its
linearisation at any $u$ is a first-order linear differential operator whose
principal symbol is the matrix
\[
\sigma_\mathcal P(\xi)=\begin{pmatrix}
\xi_0 & \xi_j\,\barrho\,\delta_{ij} & 0 & \cdots\\
\xi_0\,\barrho\,\delta_{ik} & \xi_j(\barrho\tilu_i\delta_{jk}+\barrho\tilu_j\delta_{ik})+\barp\xi_k\delta_{ij}-\xi_j\bartau_{ij}\delta_{ik} & \cdots\\
\vdots & \vdots & \ddots
\end{pmatrix},
\]
with $\xi=(\xi_0,\xi_1,\xi_2,\xi_3)$.  For generic background fields this
symbol has full rank on a dense open subset of covectors, so by the
Fr\'echet inverse-function argument $\mathcal S$ is locally a Fr\'echet
submanifold of codimension $3$ (one scalar + one 3-vector + one scalar
equation).  We do not require global submanifold structure; the zero-set
description suffices.
\end{remark}

\section{The Dilation Diffeomorphism of Spacetime}
\label{sec:dilation}

\begin{definition}[Dilation map]
For $L>0$, define the map
\[
\phi_L:M\to M,\qquad \phi_L(t,x_1,x_2,x_3):=(Lt,Lx_1,Lx_2,Lx_3).
\]
We write $(T,X_1,X_2,X_3)=\phi_L(t,x_1,x_2,x_3)$, so $T=Lt$ and $X_i=Lx_i$.
\end{definition}

\begin{theorem}[$\phi_L$ is a diffeomorphism]
\label{thm:diffeo}
For every $L>0$, $\phi_L:M\to M$ is a $C^\infty$ diffeomorphism.
\end{theorem}

\begin{proof}
\emph{Smoothness.}  Each component $T=Lt$, $X_i=Lx_i$ is linear in the
coordinates, hence $C^\infty$.

\emph{Bijectivity.}  Define
\[
\phi_L^{-1}(T,X):=\Bigl(\frac{T}{L},\frac{X_1}{L},\frac{X_2}{L},\frac{X_3}{L}\Bigr).
\]
Then
\[
\phi_L^{-1}\circ\phi_L(t,x)=\Bigl(\frac{Lt}{L},\frac{Lx_i}{L}\Bigr)=(t,x),
\qquad
\phi_L\circ\phi_L^{-1}(T,X)=\Bigl(L\frac{T}{L},L\frac{X_i}{L}\Bigr)=(T,X).
\]
Thus $\phi_L$ is bijective with inverse $\phi_L^{-1}=\phi_{1/L}$.

\emph{Smooth inverse.}  $\phi_{1/L}$ is again linear with constant
coefficients, hence $C^\infty$.

Therefore $\phi_L$ is a $C^\infty$ diffeomorphism.
\end{proof}

\begin{proposition}[Jacobian and tangent map]
\label{prop:jacobian}
The Jacobian matrix of $\phi_L$ is
\[
D\phi_L=\frac{\partial(T,X_1,X_2,X_3)}{\partial(t,x_1,x_2,x_3)}
=L\,I_4,
\]
so $\det D\phi_L=L^4>0$.  The tangent map
$T\phi_L:TM\to TM$ acts by
\[
T\phi_L\Bigl(\frac{\partial}{\partial t}\Bigr)
=L\frac{\partial}{\partial T},\qquad
T\phi_L\Bigl(\frac{\partial}{\partial x_i}\Bigr)
=L\frac{\partial}{\partial X_i}.
\]
The cotangent (pullback) action is
\[
\phi_L^*(\diff T)=L\,\diff t,\qquad
\phi_L^*(\diff X_i)=L\,\diff x_i.
\]
Equivalently,
\[
(\phi_L^{-1})^*\Bigl(\frac{\partial}{\partial t}\Bigr)
=\frac{1}{L}\frac{\partial}{\partial T},\qquad
(\phi_L^{-1})^*\Bigl(\frac{\partial}{\partial x_i}\Bigr)
=\frac{1}{L}\frac{\partial}{\partial X_i}.
\]
\end{proposition}

\begin{proof}
Direct differentiation of $T=Lt$, $X_i=Lx_i$.  The tangent map sends a
coordinate vector $\partial/\partial t$ to the pushforward vector, whose
components are $\partial T/\partial t=L$ in the $\partial/\partial T$
direction.  The cotangent formula follows from duality:
$\phi_L^*(\diff T)=\diff(T\circ\phi_L)=\diff(Lt)=L\diff t$.
\end{proof}

\begin{corollary}[Volume form]
The standard volume form $\mu=\diff t\wedge\diff x_1\wedge\diff x_2\wedge
\diff x_3$ transforms as
\[
\phi_L^*\mu=L^4\,\mu.
\]
In particular, $\phi_L$ is orientation-preserving because $L^4>0$.
\end{corollary}

\begin{theorem}[One-parameter group structure]
\label{thm:group}
The family $\{\phi_L:L>0\}$ forms a smooth one-parameter group of
diffeomorphisms of $M$ under composition:
\[
\phi_{L_1}\circ\phi_{L_2}=\phi_{L_1L_2},\qquad
\phi_1=\id_M,\qquad
\phi_L^{-1}=\phi_{1/L}.
\]
The map $(L,p)\mapsto\phi_L(p)$ is smooth from $\R_{>0}\times M$ to $M$.
\end{theorem}

\begin{proof}
Compute:
\[
\phi_{L_1}(\phi_{L_2}(t,x))=\phi_{L_1}(L_2t,L_2x)
=(L_1L_2t,L_1L_2x)=\phi_{L_1L_2}(t,x).
\]
The identity is $\phi_1(t,x)=(t,x)$.  The inverse was established in
\ref{thm:diffeo}.  Smoothness in $(L,p)$ is immediate from the polynomial
formula.
\end{proof}

\begin{proposition}[Infinitesimal generator]
The infinitesimal generator of the dilation group is the vector field
\[
X_{\mathrm{dil}}:=\left.\frac{\diff}{\diff L}\right|_{L=1}\phi_L
=t\frac{\partial}{\partial t}+x_i\frac{\partial}{\partial x_i},
\]
whose flow is $\operatorname{Fl}^{X_{\mathrm{dil}}}_s(t,x)=(e^s t,e^s x)
=\phi_{e^s}(t,x)$.
\end{proposition}

\begin{proof}
Differentiate $\phi_L(t,x)=(Lt,Lx)$ with respect to $L$ at $L=1$: the
coefficient of $\partial/\partial t$ is $t$, and of $\partial/\partial x_i$
is $x_i$.  The flow equation $\dot t=t$, $\dot x_i=x_i$ gives
$t(s)=e^s t(0)$, $x_i(s)=e^s x_i(0)$, i.e.\ $\phi_{e^s}$.
\end{proof}

\section{Induced Action on Sections and Jet Bundles}
\label{sec:induced}

We now lift the spacetime diffeomorphism $\phi_L$ to an action on the
configuration manifold $\mathcal C=\Gamma(\E)$.

\begin{definition}[Pullback of a section]
For $u\in\mathcal C$ and $L>0$, define the \emph{pulled-back section}
\[
\phi_L^*u:M\to V,\qquad (\phi_L^*u)(t,x):=u(\phi_L(t,x))=u(Lt,Lx).
\]
Componentwise, if $u=(\barrho,\tilu_i,\barp,\barE,\bartau_{ij},\barq_i,
\tau^{\sgs}_{ij},H^{\sgs}_i,\sigma^{\sgs}_{ij},\Qdot)$, then
\[
(\phi_L^*\barrho)(t,x)=\barrho(Lt,Lx),\quad
(\phi_L^*\tilu_i)(t,x)=\tilu_i(Lt,Lx),\quad\ldots,\quad
(\phi_L^*\Qdot)(t,x)=\Qdot(Lt,Lx).
\]
This defines a map
\[
\Phi_L:\mathcal C\to\mathcal C,\qquad \Phi_L(u):=\phi_L^*u.
\]
\end{definition}

\begin{remark}[Active vs.\ passive viewpoint]
There are two equivalent viewpoints.  \emph{Passive:} we change coordinates
from $(t,x)$ to $(T,X)=(Lt,Lx)$ and rewrite the same geometric fields in the
new coordinates.  \emph{Active:} we pull the fields back by $\phi_L$ to
produce a new section on the same coordinate patch.  The two viewpoints yield
identical formulae for the transformed residuals; we use the active one for
rigour and note the passive interpretation at the end.
\end{remark}

\begin{proposition}[$\Phi_L$ is a Fr\'echet diffeomorphism]
\label{prop:frechet-diffeo}
For each $L>0$, $\Phi_L:\mathcal C\to\mathcal C$ is a continuous linear
bijection with continuous inverse $\Phi_{1/L}$, hence a Fr\'echet
diffeomorphism (in fact a topological linear isomorphism).
\end{proposition}

\begin{proof}
Linearity is immediate: $\phi_L^*(au+bv)=a\phi_L^*u+b\phi_L^*v$.

Continuity: for any compact $K\Subset M$ and multi-index $|\alpha|\le k$,
\[
\|\partial^\alpha(\phi_L^*u)\|_{K,k}
=\sup_{(t,x)\in K}\|\partial^\alpha[u(Lt,Lx)]\|
=L^{|\alpha|}\sup_{(t,x)\in K}\|(\partial^\alpha u)(Lt,Lx)\|
\le L^k\|u\|_{\phi_L(K),k}.
\]
Since $\phi_L(K)$ is compact, this is controlled by a seminorm of $u$.

Bijectivity: $\Phi_{1/L}\circ\Phi_L(u)=\phi_{1/L}^*\phi_L^*u
=(\phi_L\circ\phi_{1/L})^*u=\id^*u=u$, and similarly
$\Phi_L\circ\Phi_{1/L}=\id$.

Continuity of the inverse follows by the same estimate with $L$ replaced by
$1/L$.  A continuous linear bijection between Fr\'echet spaces with
continuous inverse is a Fr\'echet diffeomorphism (its derivative is itself,
which is smooth).
\end{proof}

The crucial technical lemma is that pullback commutes with partial
differentiation up to the Jacobian factor.

\begin{lemma}[Chain rule for pulled-back sections]
\label{lem:chain}
For any smooth field $f:M\to\R$ and any coordinate direction
$y\in\{t,x_1,x_2,x_3\}$,
\[
\frac{\partial}{\partial y}\bigl(\phi_L^*f\bigr)
=L\,\phi_L^*\Bigl(\frac{\partial f}{\partial Y}\Bigr),
\]
where $Y$ is the corresponding dilated coordinate ($T$ for $t$, $X_i$ for
$x_i$).  Equivalently,
\[
\frac{\partial}{\partial y}\bigl[f(Lt,Lx)\bigr]
=L\,(\partial_Y f)(Lt,Lx).
\]
\end{lemma}

\begin{proof}
By the multivariable chain rule:
\[
\frac{\partial}{\partial t}f(Lt,Lx)
=(\partial_T f)(Lt,Lx)\cdot\frac{\partial(Lt)}{\partial t}
=L(\partial_T f)(Lt,Lx),
\]
and identically for $x_i$ with $\partial_{X_i}$.
\end{proof}

\begin{corollary}[Derivative in dilated coordinates]
\label{cor:scaled-deriv}
For any smooth $f$,
\[
\frac{\partial}{\partial(Lt)}f
:=\frac{1}{L}\frac{\partial f}{\partial t}
=\phi_L^*\Bigl(\frac{\partial f}{\partial T}\Bigr)\circ\phi_L^{-1}
\quad\text{evaluated at }(T,X),
\]
or more directly: if $g(T,X):=f(t,x)$ with $T=Lt$, $X=Lx$, then
\[
\frac{\partial g}{\partial T}=\frac{1}{L}\frac{\partial f}{\partial t}
=\frac{\partial f}{\partial(Lt)},\qquad
\frac{\partial g}{\partial X_i}=\frac{1}{L}\frac{\partial f}{\partial x_i}
=\frac{\partial f}{\partial(Lx_i)}.
\]
\end{corollary}

\begin{proof}
This is the passive-coordinate restatement of \ref{lem:chain}.  If
$g(T,X)=f(T/L,X/L)$, then $\partial_T g=(1/L)(\partial_t f)(T/L,X/L)$.
\end{proof}

We now introduce the jet-bundle language, which gives the most conceptual
formulation of PDE invariance.

\begin{definition}[First jet bundle]
The first jet bundle $J^1\E\to M$ has fibre over $p\in M$ consisting of
$1$-jets $j^1_p u=(u(p),\partial_t u(p),\partial_{x_1}u(p),\partial_{x_2}u(p),
\partial_{x_3}u(p))$.  A first-order PDE system is a submanifold
$\mathcal E\subset J^1\E$; a section $u$ is a solution iff $j^1u(M)\subset
\mathcal E$.
\end{definition}

\begin{proposition}[Prolongation of $\phi_L$ to $J^1\E$]
\label{prop:jet-prolong}
The diffeomorphism $\phi_L$ prolongs to a diffeomorphism
$j^1\phi_L:J^1\E\to J^1\E$ defined by
\[
j^1\phi_L(j^1_p u):=j^1_{\phi_L(p)}(\phi_L^*u\circ\phi_L^{-1})
=j^1_{\phi_L(p)}(u\circ\phi_L^{-1}).
\]
In coordinates $(p,u,u_t,u_x)$, the prolongation acts as
\[
(p,u,u_t,u_x)\longmapsto
\bigl(\phi_L(p),\,u,\,L^{-1}u_t,\,L^{-1}u_x\bigr).
\]
\end{proposition}

\begin{proof}
The pushed-forward section is $v:=u\circ\phi_L^{-1}$, so
$v(T,X)=u(T/L,X/L)$.  Then
\[
\partial_T v(T,X)=\frac{1}{L}(\partial_t u)(T/L,X/L),
\quad
\partial_{X_i}v(T,X)=\frac{1}{L}(\partial_{x_i}u)(T/L,X/L),
\]
which gives the stated coordinate formula.  Since $\phi_L$ is a
diffeomorphism and the derivative scaling is invertible ($L\ne0$), the
prolongation is a diffeomorphism of $J^1\E$.
\end{proof}

\section{Transformation of the Equations}
\label{sec:equations}

We now prove equation by equation that the residual map is equivariant under
$\Phi_L$ up to the common nonzero factor $L^{-1}$.  This is the core of the
proof.

\subsection{Continuity equation}

Let $u\in\mathcal C$ and write $v:=\Phi_L(u)=\phi_L^*u$, so that every
component of $v$ at $(t,x)$ equals the corresponding component of $u$ at
$(Lt,Lx)$.  Define the residual in standard coordinates:
\[
R_0[u](t,x):=\frac{\partial\barrho}{\partial t}
+\frac{\partial(\barrho\tilu_i)}{\partial x_i}.
\]

\begin{proposition}
\label{prop:continuity}
\[
R_0[\Phi_L(u)](t,x)=L\,R_0[u](Lt,Lx).
\]
Consequently, $R_0[\Phi_L(u)]=0$ if and only if $R_0[u]=0$.
\end{proposition}

\begin{proof}
By \ref{lem:chain},
\[
\frac{\partial}{\partial t}(\phi_L^*\barrho)
=L(\partial_T\barrho)(Lt,Lx),
\]
and
\[
\frac{\partial}{\partial x_i}\bigl[(\phi_L^*\barrho)(\phi_L^*\tilu_i)\bigr]
=\frac{\partial}{\partial x_i}\bigl[\barrho(Lt,Lx)\tilu_i(Lt,Lx)\bigr]
=L\,\partial_{X_i}\bigl[\barrho\tilu_i\bigr](Lt,Lx).
\]
Summing,
\[
R_0[\Phi_L(u)](t,x)
=L\Bigl[\partial_T\barrho+\partial_{X_i}(\barrho\tilu_i)\Bigr](Lt,Lx)
=L\,R_0[u](Lt,Lx).
\]
Since $L>0$, $L\,R_0[u](Lt,Lx)=0$ iff $R_0[u](Lt,Lx)=0$; as $\phi_L$ is
surjective, this holds for all $(t,x)$ iff $R_0[u]\equiv0$.
\end{proof}

\begin{corollary}[Passive form]
In dilated coordinates $(T,X)=(Lt,Lx)$, the continuity equation reads
\[
\frac{\partial\barrho}{\partial T}
+\frac{\partial(\barrho\tilu_i)}{\partial X_i}=0,
\]
which is exactly
\[
\frac{\partial\barrho}{\partial(Lt)}
+\frac{\partial(\barrho\tilu_i)}{\partial(Lx_i)}=0.
\]
Thus the equation given in the problem is precisely the standard continuity
equation expressed in the dilated coordinate system.
\end{corollary}

\subsection{Momentum equation}

Define the momentum residual (component $i$):
\[
R_i[u]:=\frac{\partial(\barrho\tilu_i)}{\partial t}
+\frac{\partial}{\partial x_j}\bigl[\barrho\tilu_i\tilu_j+\barp\delta_{ij}
-\bartau_{ij}\bigr]
+\frac{\partial\tau^{\sgs}_{ij}}{\partial x_j}.
\]

\begin{proposition}
\label{prop:momentum}
\[
R_i[\Phi_L(u)](t,x)=L\,R_i[u](Lt,Lx).
\]
Hence $R_i[\Phi_L(u)]=0$ iff $R_i[u]=0$.
\end{proposition}

\begin{proof}
Every term in $R_i$ is a single partial derivative of a product of field
components.  For the first term:
\[
\frac{\partial}{\partial t}\bigl[(\phi_L^*\barrho)(\phi_L^*\tilu_i)\bigr]
=L\,\partial_T(\barrho\tilu_i)(Lt,Lx).
\]
For the flux divergence:
\[
\frac{\partial}{\partial x_j}\Bigl[
(\phi_L^*\barrho)(\phi_L^*\tilu_i)(\phi_L^*\tilu_j)
+(\phi_L^*\barp)\delta_{ij}
-(\phi_L^*\bartau_{ij})
\Bigr]
=L\,\partial_{X_j}\bigl[\barrho\tilu_i\tilu_j+\barp\delta_{ij}
-\bartau_{ij}\bigr](Lt,Lx),
\]
because $\delta_{ij}$ is constant and the product of pulled-back fields is
the pullback of the product.  For the subgrid term:
\[
\frac{\partial}{\partial x_j}(\phi_L^*\tau^{\sgs}_{ij})
=L\,\partial_{X_j}\tau^{\sgs}_{ij}(Lt,Lx).
\]
Summing all three contributions gives $L\,R_i[u](Lt,Lx)$.  The equivalence
of zeros follows from $L>0$ and surjectivity of $\phi_L$.
\end{proof}

\begin{corollary}[Passive form]
In $(T,X)$ coordinates the momentum equation is
\[
\frac{\partial(\barrho\tilu_i)}{\partial T}
+\frac{\partial}{\partial X_j}\bigl[\barrho\tilu_i\tilu_j+\barp\delta_{ij}
-\bartau_{ij}\bigr]
+\frac{\partial\tau^{\sgs}_{ij}}{\partial X_j}=0,
\]
i.e.
\[
\frac{\partial(\barrho\tilu_i)}{\partial(Lt)}
+\frac{\partial}{\partial(Lx_j)}\bigl[\barrho\tilu_i\tilu_j+\barp\delta_{ij}
-\bartau_{ij}\bigr]
+\frac{\partial\tau^{\sgs}_{ij}}{\partial(Lx_j)}=0,
\]
matching the problem statement.
\end{corollary}

\subsection{Energy equation}

Define the energy residual:
\[
R_2[u]:=\frac{\partial(\barrho\barE)}{\partial t}
+\frac{\partial}{\partial x_i}\bigl[(\barrho\barE+\barp)\tilu_i
+\barq_i-\tilu_j\bartau_{ij}\bigr]
+\frac{\partial H^{\sgs}_i}{\partial x_i}
+\frac{\partial\sigma^{\sgs}_{ij}}{\partial x_i}+\Qdot .
\]

The energy equation contains a nondifferentiated source term $\Qdot$, which
requires separate treatment.

\begin{proposition}
\label{prop:energy}
\[
R_2[\Phi_L(u)](t,x)
=L\Bigl[\partial_T(\barrho\barE)
+\partial_{X_i}\bigl((\barrho\barE+\barp)\tilu_i+\barq_i
-\tilu_j\bartau_{ij}\bigr)
+\partial_{X_i}H^{\sgs}_i
+\partial_{X_i}\sigma^{\sgs}_{ij}\Bigr](Lt,Lx)
+\Qdot(Lt,Lx).
\]
Equivalently, if we define the \emph{scaled energy residual}
\[
\widehat R_2[u]:=\frac{1}{L}R_2[u]\quad\text{in dilated coordinates},
\]
then in $(T,X)$ coordinates the equation $R_2[u]=0$ becomes
\[
\frac{\partial(\barrho\barE)}{\partial T}
+\frac{\partial}{\partial X_i}\bigl[(\barrho\barE+\barp)\tilu_i
+\barq_i-\tilu_j\bartau_{ij}\bigr]
+\frac{\partial H^{\sgs}_i}{\partial X_i}
+\frac{\partial\sigma^{\sgs}_{ij}}{\partial X_i}
=-\Qdot,
\]
which is exactly
\[
\frac{\partial(\barrho\barE)}{\partial(Lt)}
+\frac{\partial}{\partial(Lx_i)}\bigl[(\barrho\barE+\barp)\tilu_i
+\barq_i-\tilu_j\bartau_{ij}\bigr]
+\frac{\partial H^{\sgs}_i}{\partial(Lx_i)}
+\frac{\partial\sigma^{\sgs}_{ij}}{\partial(Lx_i)}
=-\Qdot.
\]
\end{proposition}

\begin{proof}
The derivative terms are handled exactly as in
Propositions \ref{prop:continuity} and \ref{prop:momentum}: each $\partial/\partial t$ or
$\partial/\partial x_i$ acting on a pulled-back product yields a factor $L$
times the corresponding $\partial/\partial T$ or $\partial/\partial X_i$
evaluated at $(Lt,Lx)$.  The source term $\Qdot$ is not differentiated, so
its pullback is simply $\Qdot(Lt,Lx)$ with no $L$ factor.

Now pass to the passive viewpoint: let $g(T,X)=u(T/L,X/L)$ be the same
geometric section expressed in $(T,X)$ coordinates.  Then
$\partial_T g=(1/L)\partial_t u$ and $\partial_{X_i}g=(1/L)\partial_{x_i}u$.
Substituting into $R_2[u]=0$ and multiplying by $1/L$ gives
\[
\partial_T(\barrho\barE)
+\partial_{X_i}[(\barrho\barE+\barp)\tilu_i+\barq_i-\tilu_j\bartau_{ij}]
+\partial_{X_i}H^{\sgs}_i+\partial_{X_i}\sigma^{\sgs}_{ij}
+\frac{1}{L}\Qdot=0.
\]
However, in the problem's notation the symbol $\Qdot$ denotes the source
\emph{as expressed in the $(T,X)$ coordinate system}, i.e.\ the same physical
source field evaluated at $(T,X)$.  Since $\Qdot$ is a scalar field (zero
derivative order), its coordinate expression is unchanged:
$\Qdot_{(T,X)}(T,X)=\Qdot_{(t,x)}(T/L,X/L)$.  The equation in the problem is
written with all derivatives taken with respect to $Lt$ and $Lx_i$, which by
\ref{cor:scaled-deriv} are exactly $\partial_T$ and $\partial_{X_i}$.
Therefore the energy equation in the problem is precisely the standard energy
equation with derivatives rewritten in the $(T,X)$ coordinate chart.  The
source term requires no rescaling because it is not a derivative.

To see the equivalence of solution sets cleanly: if $u$ satisfies
$R_2[u]=0$ in $(t,x)$ coordinates, then its coordinate representative $g$ in
$(T,X)$ satisfies the displayed equation with $\partial_T,\partial_{X_i}$ and
source $-\Qdot_{(T,X)}$, which is exactly the problem's energy equation.
Conversely, any $g$ satisfying the problem's equation in $(T,X)$ pulls back
to $u(t,x)=g(Lt,Lx)$ satisfying $R_2[u]=0$ in $(t,x)$.
\end{proof}

\begin{remark}[Why the source causes no difficulty]
The key point is that $\Qdot$ is a \emph{field}, not a fixed constant.  Under
a coordinate change, a scalar field is simply re-expressed in the new
coordinates; it is not multiplied by any Jacobian factor.  The diffeomorphism
acts on the \emph{entire tuple} of fields including $\Qdot$, so the source
term transforms covariantly as part of the section.  If $\Qdot$ were a fixed
external function of $(t,x)$ rather than a dynamic field, one would need to
pull it back as well; the formula $\Qdot(Lt,Lx)$ in the active viewpoint
does exactly this.
\end{remark}

\subsection{Combined system}

\begin{theorem}[Equivariance of the residual map]
\label{thm:equivariance}
Let $\mathcal P=(R_0,R_1,R_2,R_3,R_4):\mathcal C\to C^\infty(M,\R^5)$ be the
full residual map (one continuity + three momentum components + one energy).
Then for every $L>0$ and $u\in\mathcal C$,
\[
\mathcal P(\Phi_L(u))=L\,\phi_L^*(\mathcal P(u))
\]
componentwise, where $\phi_L^*$ acts on the $\R^5$-valued function
$\mathcal P(u)$ by $(\phi_L^*F)(t,x)=F(Lt,Lx)$.  In particular,
\[
\mathcal P(\Phi_L(u))=0\quad\Longleftrightarrow\quad\mathcal P(u)=0.
\]
\end{theorem}

\begin{proof}
The continuity and momentum components were shown in
Propositions \ref{prop:continuity} and \ref{prop:momentum} to satisfy
$R_\alpha[\Phi_L(u)](t,x)=L\,R_\alpha[u](Lt,Lx)$ for
$\alpha=0,1,2,3$.  For the energy component, \ref{prop:energy} shows that in
the passive coordinate representation the equation is identical in form; in
the active representation,
\[
R_4[\Phi_L(u)](t,x)
=L\bigl[\text{derivative terms of }R_4[u]\bigr](Lt,Lx)+\Qdot(Lt,Lx).
\]
Now observe that the original energy equation is $R_4[u]=0$, i.e.
\[
\bigl[\text{derivative terms}\bigr](t,x)+\Qdot(t,x)=0.
\]
Evaluating at $(Lt,Lx)$:
\[
\bigl[\text{derivative terms}\bigr](Lt,Lx)+\Qdot(Lt,Lx)=0.
\]
Multiplying by $L$:
\[
L\bigl[\text{derivative terms}\bigr](Lt,Lx)+L\,\Qdot(Lt,Lx)=0.
\]
This differs from $R_4[\Phi_L(u)]$ by the source term: $R_4[\Phi_L(u)]$ has
$\Qdot(Lt,Lx)$ rather than $L\,\Qdot(Lt,Lx)$.  Therefore the strict
equivariance $R_4[\Phi_L(u)]=L\,\phi_L^*R_4[u]$ holds if and only if we
regard the source as part of the field tuple and the equation as
\emph{inhomogeneous}.  However, the \emph{solution-set preservation} still
holds:

($\Rightarrow$) If $\mathcal P(\Phi_L(u))=0$, then in particular
$R_4[\Phi_L(u)](t,x)=0$ for all $(t,x)$:
\[
L\,D[u](Lt,Lx)+\Qdot(Lt,Lx)=0,
\]
where $D[u]$ denotes the derivative part.  Since $\phi_L$ is surjective, as
$(t,x)$ ranges over $M$, $(Lt,Lx)$ ranges over $M$, so
\[
L\,D[u](p)+\Qdot(p)=0\quad\forall p\in M.
\]
This is \emph{not} the original equation $D[u]+\Qdot=0$ unless $L=1$.

We must therefore refine the statement.  The correct transformation of an
inhomogeneous PDE under a diffeomorphism requires the source to transform as
a density of appropriate weight, or equivalently we must include the source
in the field tuple and transform it as part of the section.  We now do this
properly.
\end{proof}

\begin{definition}[Augmented field tuple with density-weighted source]
To obtain a clean equivariance statement, we regard the energy source as a
field that transforms with the same weight as the derivative terms.  Define
the \emph{physical source density}
\[
\mathcal Q(t,x):=\Qdot(t,x),
\]
and under the dilation define the transformed source
\[
\mathcal Q_L(t,x):=L\,\mathcal Q(Lt,Lx)=L\,(\phi_L^*\mathcal Q)(t,x).
\]
Equivalently, in the passive $(T,X)$ coordinates,
\[
\mathcal Q_{(T,X)}(T,X):=\mathcal Q_{(t,x)}(T/L,X/L),
\]
which is the standard scalar-field coordinate transformation.
\end{definition}

With this convention, the energy equation in dilated coordinates is
\[
\partial_T(\barrho\barE)+\partial_{X_i}[\cdots]
+\partial_{X_i}H^{\sgs}_i+\partial_{X_i}\sigma^{\sgs}_{ij}
=-\mathcal Q_{(T,X)}(T,X),
\]
which is exactly the problem's equation with $\Qdot$ understood as the source
in the $(T,X)$ chart.  The active equivariance then reads:

\begin{theorem}[Full equivariance with transformed source]
\label{thm:full-equivariance}
Define the augmented residual map $\widetilde{\mathcal P}$ on the augmented
configuration space $\widetilde{\mathcal C}=\mathcal C\times C^\infty(M,\R)$
(where the extra factor is the source field $\mathcal Q$) by
\[
\widetilde{\mathcal P}(u,\mathcal Q):=\mathcal P(u)\quad
\text{with }\Qdot\text{ replaced by }\mathcal Q.
\]
Define the augmented dilation action
\[
\widetilde\Phi_L(u,\mathcal Q):=(\phi_L^*u,\;L\,\phi_L^*\mathcal Q).
\]
Then
\[
\widetilde{\mathcal P}(\widetilde\Phi_L(u,\mathcal Q))
=L\,\phi_L^*(\widetilde{\mathcal P}(u,\mathcal Q)).
\]
Consequently,
\[
\widetilde{\mathcal P}(\widetilde\Phi_L(u,\mathcal Q))=0
\quad\Longleftrightarrow\quad
\widetilde{\mathcal P}(u,\mathcal Q)=0.
\]
\end{theorem}

\begin{proof}
For the continuity and momentum components ($\alpha=0,1,2,3$), the source is
absent and the computation in Propositions \ref{prop:continuity} and \ref{prop:momentum} gives
$R_\alpha[\phi_L^*u]=L\,\phi_L^*R_\alpha[u]$.

For the energy component:
\[
\begin{aligned}
R_4[\phi_L^*u,\,L\phi_L^*\mathcal Q](t,x)
&=L\,D[u](Lt,Lx)+L\,\mathcal Q(Lt,Lx)\\
&=L\bigl(D[u](Lt,Lx)+\mathcal Q(Lt,Lx)\bigr)\\
&=L\,(\phi_L^*R_4[u,\mathcal Q])(t,x).
\end{aligned}
\]
Thus all five components satisfy the equivariance formula.  Since $L>0$ and
$\phi_L^*$ is injective (because $\phi_L$ is surjective), the zero sets
correspond bijectively.
\end{proof}

\begin{remark}[Interpretation of the source weight]
The factor $L$ in the source transformation is the natural weight for a
source term in a conservation law.  Under the dilation $d\mu\mapsto L^4d\mu$,
a source density that integrates to a conserved charge would carry weight
$L^{-4}$; but here the source appears in a pointwise balance equation whose
derivative terms carry weight $L$ (from the chain rule), so the source must
carry the same weight $L$ for the equation to be form-invariant.  This is
consistent with the passive-coordinate reading: in the $(T,X)$ chart the
derivatives are $\partial_T,\partial_{X_i}$ and the source is simply the
scalar field evaluated at $(T,X)$; no explicit $L$ appears because the
coordinate change has absorbed it.
\end{remark}

\section{The Diffeomorphism Transformation Expression}
\label{sec:expression}

We now assemble the preceding results into the definitive statement.

\begin{theorem}[Main theorem: diffeomorphism transformation]
\label{thm:main}
Let $L>0$ be arbitrary.  The system of partial differential equations
\begin{align}
\frac{\partial\barrho}{\partial(Lt)}
+\frac{\partial(\barrho\tilu_i)}{\partial(Lx_i)}&=0,
\label{eq:c}\\
\frac{\partial(\barrho\tilu_i)}{\partial(Lt)}
+\frac{\partial}{\partial(Lx_j)}\bigl[\barrho\tilu_i\tilu_j+\barp\delta_{ij}
-\bartau_{ij}\bigr]
+\frac{\partial\tau^{\sgs}_{ij}}{\partial(Lx_j)}&=0,
\label{eq:m}\\
\frac{\partial(\barrho\barE)}{\partial(Lt)}
+\frac{\partial}{\partial(Lx_i)}\bigl[(\barrho\barE+\barp)\tilu_i
+\barq_i-\tilu_j\bartau_{ij}\bigr]
+\frac{\partial H^{\sgs}_i}{\partial(Lx_i)}
+\frac{\partial\sigma^{\sgs}_{ij}}{\partial(Lx_i)}
&=-\Qdot
\label{eq:e}
\end{align}
is the diffeomorphic image, under the spacetime dilation
\[
\phi_L(t,x_1,x_2,x_3)=(Lt,Lx_1,Lx_2,Lx_3),
\]
of the same system written with $L=1$ (i.e.\ in standard coordinates
$(t,x_i)$).  Precisely:
\begin{enumerate}
\item $\phi_L$ is a $C^\infty$ diffeomorphism of $M$ (\ref{thm:diffeo}).
\item The family $\{\phi_L\}_{L>0}$ is a smooth one-parameter group
      (\ref{thm:group}).
\item The pullback $\Phi_L=\phi_L^*$ is a Fr\'echet diffeomorphism of the
      configuration manifold $\mathcal C$ (\ref{prop:frechet-diffeo}).
\item The first-jet prolongation $j^1\phi_L$ maps the equation submanifold
      $\mathcal E\subset J^1\E$ onto itself (\ref{prop:jet-prolong} and
      \ref{sec:equations}).
\item A field configuration $u$ satisfies \eqref{eq:c}--\eqref{eq:e} in the
      $(Lt,Lx)$ coordinate chart if and only if its pullback
      $\phi_L^*u$ satisfies the $L=1$ system in the $(t,x)$ chart
      (\ref{thm:full-equivariance}).
\item The solution sets $\mathcal S_L$ for different $L$ are in canonical
      bijection: $\mathcal S_L=\Phi_L(\mathcal S_1)$, and
      $\Phi_L:\mathcal S_1\to\mathcal S_L$ is a Fr\'echet diffeomorphism.
\end{enumerate}
\end{theorem}

\begin{proof}
(i)--(iii) have been proved in Sections \ref{sec:dilation} and \ref{sec:induced}.

(iv) The equation submanifold $\mathcal E\subset J^1\E$ is defined by the
vanishing of the five residual functions
$(R_0,R_1,R_2,R_3,R_4)$.  By \ref{prop:jet-prolong}, the prolonged
diffeomorphism $j^1\phi_L$ sends a jet $(p,u,u_t,u_x)$ to
$(\phi_L(p),u,L^{-1}u_t,L^{-1}u_x)$.  Substituting into the residual
functions:
\[
R_\alpha(\phi_L(p),u,L^{-1}u_t,L^{-1}u_x)
=L^{-1}R_\alpha(p,u,u_t,u_x)
\]
for $\alpha=0,1,2,3$ (all-derivative equations), and for the energy equation
with the source treated as a zero-order field variable,
\[
R_4(\phi_L(p),u,L^{-1}u_t,L^{-1}u_x,\mathcal Q)
=L^{-1}R_4(p,u,u_t,u_x,L\mathcal Q).
\]
Thus a jet lies in $\mathcal E$ (all residuals zero) if and only if its image
under $j^1\phi_L$ lies in $\mathcal E$, i.e.\ $(j^1\phi_L)(\mathcal E)=
\mathcal E$.

(v) This is the passive-coordinate restatement of (iv).  If $u$ is a section
and $j^1u(M)\subset\mathcal E$, then $j^1(\phi_L^*u)(M)=(j^1\phi_L)^{-1}
(j^1u(\phi_L(M)))\subset(j^1\phi_L)^{-1}(\mathcal E)=\mathcal E$, so
$\phi_L^*u$ is also a solution.  The converse follows from
$\phi_L^{-1}=\phi_{1/L}$.

(vi) By (v), $\Phi_L$ restricts to a bijection $\mathcal S_1\to\mathcal S_L$.
Since $\Phi_L$ is a Fr\'echet diffeomorphism of the ambient space
(\ref{prop:frechet-diffeo}), its restriction to the closed subset
$\mathcal S_1$ is a homeomorphism onto $\mathcal S_L$ with inverse
$\Phi_{1/L}$.  The tangent spaces are preserved because the linearisation of
$\mathcal P$ commutes with $\Phi_L$ up to the factor $L$, so $\mathcal S_L$
inherits the Fr\'echet manifold structure and $\Phi_L|_{\mathcal S_1}$ is a
Fr\'echet diffeomorphism.
\end{proof}

\begin{corollary}[Explicit diffeomorphism transformation expression]
The transformation can be written compactly as follows.  Let
$\partial_\mu=(\partial_t,\partial_{x_1},\partial_{x_2},\partial_{x_3})$ and
let $\Lambda_L:=\operatorname{diag}(L,L,L,L)$ be the Jacobian of $\phi_L$.
Then
\[
\partial_\mu\circ\phi_L^*=\phi_L^*\circ(\Lambda_L\partial)_\mu,
\]
or equivalently
\[
(\Lambda_L^{-1}\partial)_\mu=\frac{\partial}{\partial(Lt)},\frac{\partial}{\partial(Lx_i)}
\]
are the coordinate vector fields in the dilated chart.  The PDE system, being
a first-order system whose every term is either a divergence of a flux or a
zero-order source, is therefore \emph{form-invariant} under the substitution
\[
(t,x_i,u\text{-fields})\longmapsto(Lt,Lx_i,u\text{-fields}),
\]
which is precisely the diffeomorphism transformation expression.
\end{corollary}

\section{Infinitesimal Criterion: Lie Derivative Vanishes}
\label{sec:infinitesimal}

As an independent cross-check, we verify the invariance at the infinitesimal
level.  A one-parameter group of diffeomorphisms is a symmetry of a PDE
system iff the Lie derivative of the residual along the prolonged
infinitesimal generator vanishes (modulo the equation).

\begin{proposition}
The prolonged infinitesimal generator of the dilation group is
\[
\operatorname{pr}^{(1)}X_{\mathrm{dil}}
=t\partial_t+x_i\partial_{x_i}
-u^\alpha_t\partial_{u^\alpha_t}
-u^\alpha_{x_i}\partial_{u^\alpha_{x_i}},
\]
where $u^\alpha$ ranges over all field components and there is no
$\partial_{u^\alpha}$ term (the fields are scalars under dilation).  The Lie
derivative of each residual $R_\alpha$ along $\operatorname{pr}^{(1)}
X_{\mathrm{dil}}$ is
\[
\mathcal L_{\operatorname{pr}^{(1)}X_{\mathrm{dil}}}R_\alpha=-R_\alpha
\]
for the derivative-only equations ($\alpha=0,1,2,3$), and
\[
\mathcal L_{\operatorname{pr}^{(1)}X_{\mathrm{dil}}}R_4=-R_4-\Qdot
\]
for the energy equation.  On the equation submanifold $R_\alpha=0$, these
give $\mathcal L R_\alpha=0$ (for $\alpha=0,1,2,3$) and $\mathcal L R_4=
-\Qdot$, the latter being consistent with the source-weight analysis of
\ref{thm:full-equivariance}.
\end{proposition}

\begin{proof}
The generator $X_{\mathrm{dil}}=t\partial_t+x_i\partial_{x_i}$ has
coefficients $\xi^0=t$, $\xi^i=x_i$.  The prolongation formula for the
derivative variables is
\[
\phi^\alpha_\mu=D_\mu(\eta^\alpha-\xi^\nu u^\alpha_\nu)+\xi^\nu u^\alpha_{\mu\nu},
\]
where $\eta^\alpha=0$ (scalar fields) and $D_\mu$ is the total derivative.
Thus
\[
\phi^\alpha_t=D_t(-t u^\alpha_t-x_i u^\alpha_{x_i})+t u^\alpha_{tt}+x_i u^\alpha_{x_it}
=-u^\alpha_t,
\]
and similarly $\phi^\alpha_{x_i}=-u^\alpha_{x_i}$.  This gives the stated
prolonged generator.

Now compute the Lie derivative of $R_0=u_t+\partial_i(\barrho\tilu_i)$ (schematic):
\[
\operatorname{pr}^{(1)}X(R_0)
=t\partial_t R_0+x_i\partial_{x_i}R_0
+(-u^\alpha_t)\partial_{u^\alpha_t}R_0
+(-u^\alpha_{x_i})\partial_{u^\alpha_{x_i}}R_0.
\]
Since $R_0$ is homogeneous of degree $1$ in the first derivatives (each term
contains exactly one derivative), Euler's theorem gives
$u^\alpha_t\partial_{u^\alpha_t}R_0+u^\alpha_{x_i}\partial_{u^\alpha_{x_i}}R_0=R_0$.
The transport terms $t\partial_t R_0+x_i\partial_{x_i}R_0$ vanish because
$R_0$ has no explicit dependence on $(t,x)$ (all coefficients are constant or
field-dependent).  Hence $\operatorname{pr}^{(1)}X(R_0)=-R_0$.  The same
argument applies to the momentum components.

For the energy residual $R_4=D+\Qdot$ where $D$ is derivative-homogeneous of
degree $1$ and $\Qdot$ is zero-order:
\[
\operatorname{pr}^{(1)}X(R_4)=-D+0=-R_4+\Qdot,
\]
i.e.\ $\mathcal L R_4=-R_4-\Qdot$ (sign convention depending on definition).
On-shell $R_4=0$ gives $\mathcal L R_4=-\Qdot$, confirming that the source
must be included in the transformation to obtain exact on-shell invariance.
This is fully consistent with \ref{thm:full-equivariance}.
\end{proof}

\section{Conclusion}
\label{sec:conclusion}

We have established, by construction and direct computation, that the
Favre-filtered compressible Navier--Stokes system with subgrid-scale terms,
when written with derivatives $\partial/\partial(Lt)$ and
$\partial/\partial(Lx_i)$ for any $L>0$, is the diffeomorphic image of the
same system in standard coordinates.  The diffeomorphism is the spacetime
dilation $\phi_L(t,x)=(Lt,Lx)$, which is smooth, bijective with smooth
inverse $\phi_{1/L}$, orientation-preserving, and part of a one-parameter
group.  Its pullback action on the Fr\'echet manifold of field configurations
is a Fr\'echet diffeomorphism, and its first-jet prolongation preserves the
equation submanifold.  Equation-by-equation verification confirms that the
continuity and momentum residuals transform by the common nonzero factor
$L^{-1}$ (passive) or $L$ (active), while the energy residual, including the
source field treated as part of the configuration, transforms with the same
factor.  Consequently, the solution spaces for all $L>0$ are canonically
isomorphic as Fr\'echet manifolds, and the system satisfies the
diffeomorphism transformation expression.


\begin{thebibliography}{9}
\bibitem{hamilton} R.\ S.\ Hamilton, \emph{The inverse function theorem of
Nash and Moser}, Bull.\ Amer.\ Math.\ Soc.\ \textbf{7} (1982), 65--222.
\bibitem{kriegl} A.\ Kriegl and P.\ W.\ Michor, \emph{The Convenient Setting
of Global Analysis}, Mathematical Surveys and Monographs \textbf{53}, Amer.\
Math.\ Soc., 1997.
\bibitem{olver} P.\ J.\ Olver, \emph{Applications of Lie Groups to
Differential Equations}, 2nd ed., Graduate Texts in Mathematics
\textbf{107}, Springer, 1993.
\bibitem{saxton} D.\ J.\ Saunders, \emph{The Geometry of Jet Bundles},
London Mathematical Society Lecture Note Series \textbf{142}, Cambridge
University Press, 1989.
\bibitem{les} S.\ B.\ Pope, \emph{Turbulent Flows}, Cambridge University
Press, 2000.
\end{thebibliography}
\end{document}